\documentclass[a4paper,reqno,11pt]{amsart}
	\usepackage[latin1]{inputenc}
    \usepackage[T1]{fontenc}
    \usepackage[english]{babel}
    \usepackage{appendix}
    \usepackage[english]{minitoc}
    \usepackage[nice]{nicefrac}
	\usepackage[top=2cm, bottom=2cm, left=2cm,right=2cm]{geometry}
\numberwithin{equation}{section}
\makeatletter\@addtoreset{equation}{section} %-------------------------------------------------
	\usepackage{amsmath}
    \usepackage{amssymb}
    \usepackage{amsbsy} %\usepackage{amsthm}
    \usepackage{latexsym, amsfonts}
    \usepackage{graphics}
    \usepackage{ulem}
    \usepackage{hhline}
    \usepackage{dsfont}
    \usepackage{mathrsfs}
    \usepackage{color}
    \usepackage{fancyhdr}
    \usepackage{rotating}
    \usepackage{fancybox}
    \usepackage{colortbl}
    \usepackage{pifont}
    \usepackage{setspace}
    \usepackage{enumerate}
    \usepackage{multicol}
    \usepackage{varioref}
    \usepackage{lmodern}
    \usepackage{textcomp}
    \usepackage{euscript}
    \usepackage[pdftex]{hyperref}
			\newtheorem{theorem}{Theorem}[section]

			\newtheorem{proposition}[theorem]{Proposition}

			\newtheorem{remark}[theorem]{Remark}

\newcommand{\C}{\mathbb C}    \newcommand{\R}{\mathbb R}   \newcommand{\Z}{\mathbb Z} 	 

\newcommand{\Di}{\mathbb{D}}

\newcommand{\gm}{\gamma}
\newcommand{\HolD}{\mathcal{H}ol(\Di)}

\newcommand{\HnDi}{\mathcal{H}^{2}}
\newcommand{\AnTDi}{\mathcal{A}^{2}}
\newcommand{\HnWDi}{H^{2,\gamma}}
\newcommand{\AnTWDi}{\mathcal{A}^{2,\gamma}}%{\mathcal{A}^{0,2,\gamma}}

\newcommand{\norm}[1]{{\left\|{#1}\right\|}}    
   \newcommand{\scal}[1]{{\left\langle{#1}\right\rangle}}
    \newcommand{\bz}{\overline{z}}
  \newcommand{\bw}{\overline{w}}
   
\newcommand{\minmn}{m\wedge n}  
 
\begin{document}

\title[]{A note on weighted poly-Bergman spaces}
\author[R. El Harti, A. Elkachkouri, A. Ghanmi]{R. El Harti,, A. ElKachkouri, A. Ghanmi}
 
 \email{rachid.elharti@uhp.ac.ma}
 \email{elkachkouri.abdelatif@gmail.com} 
 \email{allal.ghanmi@um5.ac.ma} 
 \address{Analysis, P.D.E. and Spctral Geometry, Lab. M.I.A.-S.I., CeReMAR, Department of Mathematics, P.O. Box 1014, Faculty of Sciences, Mohammed V University in Rabat, Morocco}
 \address{Department of Mathematics and Computer Sciences, Faculty of Sciences and
 	Techniques,
 	University Hassan I, BP 577 Settat, Morocco}

\begin{abstract}
 We introduce à la Vasilevski the weighted poly-Bergman spaces in the unit disc and provide concrete orthonormal basis and give close expression of their reproducing kernel. The main tool in the description if these spaces is the so-called disc polynomials that form an orthogonal basis of the whole weighted Hilbert space.
\end{abstract}

 
%Mathematics Subject Classification (2010) Primary 46E22 · 32A25;%Secondary 31A10 · 30G30 
% AMS Subject Classification: 33C45, 33C50, 33C80.
%30E20  	Integration, integrals of Cauchy type, integral representations of analytic functions

\maketitle

\section{Introduction and statement of main results} %%%%%%%%%%%%%%%%%%%%%%%%%%%%%%%%%%%%%%%%%%%%%%%%%%%%%%%%%%%%%%%%%%%%%%%%%%%%%%%%%%%%%%%%%%%%%%%%%%%%%%%%%%%%%%%%%%%%%%%%

The present note is concerned with the weighted poly-Bergman space in the unit disc $\Di=\{z\in \C; \, |z|^2=z\bz<1\}$ defined as the specific generalization of weighted Bergman spaces to polyanalytic setting.
 Recall that polyanalytic functions of order $n$ (or $n$-analytic), which are natural generalization of holomorphic functions,  are the solutions of the generalized Cauchy equation $\partial^{n}_{\bz} f=0$,  
where  $\partial_{\bz} $ denotes the Wirtinger differential operator  
$$  \partial_{\bz} =\frac{\partial}{\partial \bz}  =\frac{1}{2} \left( \frac{\partial}{\partial x} + i \frac{\partial}{\partial y}\right) .$$
Such functions, bianalytic ones $ f = \bz \varphi_1+ \varphi_0$ were  appeared  first in the study of plane elasticity by Kolossov's  \cite{Kolossov1908papers}.
Since then, they have been extensively studied 
%by the Russian mathematicians 
and  have numerous applications in mechanics and mathematical physics 
\cite{Muskhelishvili1968,Wendland1979,HaimiHedenmalm2013}.
Recently, they have found interesting applications in signal processing, time--frequency analysis and
wavelet theory \cite{Grochenig2001,GrochenigLyubarskii2009,AbreuFeichtinger2014}.
For a complete survey on these functions one can refer to \cite{Balk1991}.

The extension of the classical Bergman space  $\mathcal{A}^{2}(\Di) := L^{2}\left(\Di; dxdy\right) \cap \HolD$  
to the context of polyanalytic functions was proposed by Koshelev \cite{Koshelev1977}, who proved that the set $\HnDi_n(\Di)$ of $n+1$-analytic complex-valued functions in $\Di$ belonging to  $L^{2}\left(\Di; dxdy\right)$ is a reproducing kernel Hilbert space for which the functions
\begin{align}\label{koshelevemp}
e_{m,p} (z) :=   \frac{\sqrt{m+p+1}}{\sqrt{\pi}(m+p)!} \frac{\partial^{m+p}}{\partial z^p \partial \bz^m}\left( (|z|^2-1)^{m+p} \right), 
\end{align}
for varying $p=0,1,2,\cdots, n$ and $m=0,1,2,\cdots$, form a complete  orthonormal polynomial system.
Subsequently the so-called true poly-Bergman spaces (according to Vasilevski's terminology)
are particular subspaces $\AnTDi_{n}(\Di)$ giving rise to piecewise decomposition of $\HnDi_n(\Di)$.
Namely,  \cite{Ramazanov1999}
%\textcolor{red}
	$$\AnTDi_{n}(\Di)=\left\lbrace f:f(z)= \partial^{n}_z \left((1-z\overline{z})^{n}F(z) \right), F\in \mathcal{A}(\Di), f\in \HnDi_{n}(\Di)   \right\rbrace $$
with orthogonal basis given by the system of polynomials $\{e_{m,n}; m=0,1,\cdots,\}$. Its direct sum forms the full $\HnDi_{n}(\Di)$ space,
$\HnDi_{n}(\Di)=\AnTDi_{0}(\Di)\oplus\AnTDi_{1}(\Di)\oplus\ldots\oplus\AnTDi_{n}(\Di)$ with  $\AnTDi_{0}(\Di)=\mathcal{A}^{2}(\Di)=\HnDi_{0}$.
The involved operator {$D_{n} f :=  \partial^{n}_{z}((1-|z|^{2})^{n}u(z))$} defines a bounded and boundedly invertible operator from $\mathcal{A}^{2}(\Di)$  onto $\AnTDi_{n}(\Di)$ (see e.g. \cite[Theorem 1]{Ramazanov1999}).
For more details on these spaces, one can refer  e.g.  \cite{Koshelev1977,Ramazanov1999,Vasilevski2000,RozenblumVasilevski2019}.

The weighted version $\mathcal{A}^{2,\gm}(\Di)$ of the Bergman space  is defined as the closed subspace of  $L^{2,\gamma}(\Di):=L^{2}\left(\Di,d\mu_{\gamma}\right)$, the Hilbert space of all complex-valued functions in $\Di$ with the norm induced from the scalar product
%$$ \scal{ f,g}_\gm := \int_{\Di} f(\xi) \overline{g(\xi)} d\mu_\gm(\xi), \quad d\mu_\gm(z) = \left(1- |z|^{2}\right)^{\gm} dxdy $$
$$ \scal{ f,g}_\gm := \int_{\Di} f(\xi) \overline{g(\xi)} d\mu_\gm(\xi), $$
where  respect to the weighted (probability) measure
%; up to a multiplicative constant, )
\begin{align} \label{weightg} d\mu_\gamma(z) =  (1-|z|^2)^{\gamma} dxdy; \, \, z=x+iy \in \Di.
\end{align} 
It is a reproducing kernel Hilbert space with kernel given by 
\begin{align} \label{RK-wBs}
K^{\gm}(z,w)= \frac{\gm+1}{\pi \left( 1 -  z \bw\right)^{\gm+2}} .
\end{align}

We define the weighted poly-Bergman space $\HnWDi_n(\Di)$  as the space of polyanalytic functions of order $n+1$ that are square integrable on $\Di$ with respect to the radial weight function $(1-|z|^2)^{\gamma}$, so that $\HnWDi_0(\Di)=\mathcal{A}^{2,\gm}(\Di)$. To wit 
$$\HnWDi_n(\Di) = \{ F \in L^{2,\gamma}(\Di); \, \partial_{\bz}^{n+1} F =0  \}.  $$  
The concrete description of $\HnWDi_n(\Di)$ made appeal to the so-called disc polynomials defined by
\begin{align}\label{DiscePol}
\mathcal{R}_{m,n}^\gamma (z,\bar z)
= \frac{z^m \bz^n}{|z|^{2(m\wedge n)}}P^{(\gamma,|m-n|)}_{m\wedge n} (2|z|^2-1),
\end{align} 
where $m\wedge n=\min(m,n)$ and $P_{m}^{(\alpha,\beta)}(x)$ denotes the real Jacobi polynomials normalized so that $P_{m}^{(\alpha,\beta)}(1)=1$. 
See Section 2 for a brief review of some basic facts on  $\mathcal{R}_{m,n}^\gamma$. The basic properties of the space $A^{2,\gamma}_n(\Di)$ are presented in Section 3. Mainly, we prove the following theorem. 

\begin{theorem}\label{thm1}
	The weighted poly-Bergman spaces
	$\HnWDi_n(\Di) $ are closed subspaces of $L^{2,\gamma}(\Di)$. Moreover, they coincide with those spanned by the disc polynomials $\mathcal{R}^{\gamma}_{m,k}$ for varying $m=0,1,2, \cdots,$ and $k=0,1,\cdots,n$.
%	\begin{align}\label{wpolyBerg}
%	\HnWDi_n(\Di) = span\{\mathcal{R}^{\gamma}_{m,k}; \quad k=0,1,\cdots,n: \, m=0,1,2, \cdots \} .   
%	\end{align} 
\end{theorem}

The next result concerns the so-called the $n$-th true weighted true poly-Bergman space defined as  $\mathcal{A}_{0}^{2,\gamma}(\Di) = \AnTWDi(\Di)$ and 
$$\AnTWDi_{n}(\Di)= \HnWDi_n(\Di) \ominus \HnWDi_{n-1}(\Di); n\geq 1,$$
 and gives  the explicit closed expression of its  reproducing kernel function $K^{\gm}_n(z,w)$ 
in terms of the Gauss hypergeometric function
$$ {_2F_1}\left( \begin{array}{c} a , b \\ c \end{array}\bigg | x \right) = \sum_{j=0}^{\infty} \frac{(a)_j (b)_j}{(c)_j} \frac{x^j}{j!}.
$$

\begin{theorem}\label{thm2}
	
	The following assertions hold trues
	\begin{enumerate}
		\item[(i)] The family of functions $\mathcal{R}^{\gamma}_{m,n}$, for varying $m=0,1,2, \cdots$ form an orthogonal basis of $\AnTWDi_{n}(\Di)$.
		
		\item[(ii)] The spaces $\AnTWDi_{n}(\Di)$ form an orthogonal sequence of reproducing kernel Hilbert spaces in $L^{2,\gamma}(\Di)$, and we have  	$$L^{2,\gamma}(\Di) =  \bigoplus_{n=0}^\infty \AnTWDi_{n}(\Di) \quad \mbox{and} \quad \HnWDi_n(\Di) = \bigoplus_{k=0}^n \mathcal{A}_{k}^{2,\gamma}(\Di).$$
		
		\item[(iii)]  We have	 
		\begin{align} \label{RK-wBsm}
		K^{\gm}_n(z,w) &= \frac{(\gamma+n+1) }{\pi n! (\gamma+1)_n } \left(1-|z|^2 \right)^{-\gamma} \left(1-|w|^2 \right)^{-\gamma} \\
		& \times
		\partial_z^n\partial_{\bw}^n \left\{ \left(1-|z|^2 \right)^{\gamma+n} \left(1-|w|^2 \right)^{\gamma+n} {_2F_1}\left( \begin{array}{c} \gamma+n+2, \gamma+1 \\ \gamma+n+1 \end{array}\bigg | z\bw \right) \right\} . \nonumber
		\end{align} 
	\end{enumerate} 
\end{theorem}

\begin{remark}
	We have 
	$\AnTWDi_n(\Di) = \overline{Span  \left\{ \mathcal{R}^{\gamma}_{m,n} ; \, m=0,1,2, \cdots \right\}}^{L^{2,\gamma}(\Di)}$
	and the  sequential characterization is given by 
	\begin{align}\label{wpolyBerg}
	\AnTWDi_{n}(\Di) = \left\{ \sum_{m=0}^{+\infty}\alpha_{m}\mathcal{R}^{\gamma}_{m,n}(\xi,\overline{\xi}); \quad  \sum_{m=0}^\infty \frac{m! }{(\gamma+1+m+n)(\gamma+1)_m} |\alpha_n|^2 < +\infty \right\} . 
	\end{align}
\end{remark}

\begin{remark}
	The spaces $\mathcal{A}_{n}^{2,0}(\Di)$, for $\gamma=0$, are exactly the true poly-Bergman spaces \cite{Ramazanov1999,Vasilevski2000}, and  the polynomials in $e_{m,p}$ in \eqref{koshelevemp} 
	reduce further to $\mathcal{R}^{0}_{p,m}$. In this case, we recover the closed  expression obtained in \cite[Theorem 2]{Koshelev1977} and \cite[Theorem 3]{Ramazanov1999} for the reproducing kernel of $\HnWDi_{n}(\Di)$ and $\AnTWDi_{n}(\Di)$, respectively.
\end{remark}

\begin{remark}
	For $n=0$, we recover the closed  expression of the classical weighted Bergman spaces since 
	$$ {_2F_1}\left( \begin{array}{c} \gamma+2, \gamma+1 \\ \gamma+1 \end{array}\bigg | z\bw \right) = (1- z\bw)^{-\gamma-2}.$$ 
\end{remark}

\begin{remark}
	The orthogonal projection from $ L^{2,\gamma}(\Di)  $ onto  $ \AnTWDi_{n}(\Di) $ is given by the integral operator
	$$ \mathcal{P}^{0}_{n}(f)(\zeta)=\int_{\Di} f(z)
	K^{\gm}_n(z,\zeta) d\mu_\gm(z),$$
	while the one of $L^{2,\gamma}(\Di) $ onto  $\HnWDi_{n}(\Di)$ is given by 
	$$\mathcal{P}_{n }(f)(\zeta)=\int_{\Di} f(z)\left( \sum^{n}_{k=1}  K^{\gm}_n(z,\zeta) \right) d\mu_\gm(z).$$
\end{remark}

\section{Preliminaries on disc polynomials}

The disc polynomials $\mathcal{R}_{m,n}^\gamma (z,\bar z)$; $\gamma > -1$, defined through \eqref{DiscePol}, are given explicitly by
\begin{align} \label{expZ}
\mathcal{R}_{m,n}^{\gamma}(z,\bar{z})
= m!n!\sum_{j=0}^{\minmn }
\frac{(-1)^{j} (1 -  z\bz)^{j}} {j! (\gamma+1)_j } \frac{ z^{m-j}}{(m-j)! } \frac{\bz^{n-j} }{(n-j)!},
	\end{align}
	where $(a)_k=a(a+1) \cdots (a+k-1)$ denotes the Pochhammer symbol. 
	For the hypergeometric representation of $\mathcal{R}_{m,n}^\gamma$ in terms of the Gauss-hypergeometric function ${_2F_1}$, one can refer to e.g. \cite[p. 137]{Wunsch05},  \cite[p. 692]{Dunkl83}, \cite[p. 535]{Dunkl84}, or also \cite{Aharmim2015,ElHGhIn2015Arxiv}.
 Clearly, they are polynomials of the two conjugate complex variables $z=x+iy\in \Di$ and $\overline{z}=x-iy$; $x,y\in\R$, of degree $m$ and $n$, respectively. The suggested definition  
agrees with the ones provided by Koornwinder \cite{Koornwinder75,Koornwinder78}, Dunkl \cite{Dunkl83,Dunkl84}, and W\"unsche \cite{Wunsch05}.
 The limit case of $\gamma=-1$ leads to the so-called scattering polynomials that has emerged in the context of wave propagation in layered media \cite{Gibson2014}, while for $\gamma=0$, they turn out to be related to the radial Zernike polynomials $R^\nu_k(x)$, introduced by Zernike  \cite{Zernike34} in his framework on optical problems involving telescopes and microscopes, and playing an important role in expressing the wavefront data in optical tests %\cite{}
 and in studying diffraction problems \cite{ZernikeBrinkman35}. More exactly, we have 
$$\mathcal{R}_{m,n}^0 (z,\bar z) =   (m+n)! e^{i(n-m)\arg z}  R^{n-m}_{m+n}(\sqrt{z\bz}); \, m\leq n.$$
	The orthogonality property 
 \begin{align} \label{OrthDiscPol}
 \int_{\Di} \mathcal{R}_{m,n}^{\gamma}(z,\overline{z}) \overline{\mathcal{R}_{j,k}^{\gamma}(z,\overline{z})}  d\mu_\gamma(z) = d_{m,n}^\gamma \delta_{m,j}\delta_{n,k}, 
 	\end{align}
 with $d\mu_\gm$ is as in \eqref{weightg}, follows by straightforward computation from its analogue  the Jacobi polynomials. The involved constant $d_{m,n}^\gamma$ is explicitly given by
 \begin{align} \label{Cst}
 d_{m,n}^\gamma = \frac{\pi  m! n!}{(\gamma+1+m+n)(\gamma+1)_m(\gamma+1)_n}.
 	\end{align}
For ulterior use, {mainly in proving the convergence of the relevant series}, we recall the following estimate \cite{Dunkl84,Kanjin2013}, $|\mathcal{R}_{m,n}^{\gamma}(z,\bar z)| \leq 1$ which holds true for every nonnegative integers $m$ and $n$,  real $\gamma >-1$ and $z\in \Di$.

 The proof of our main results relies essentially in the fact that $\mathcal{R}_{m,n}^\gamma (z,\bar z)$ form an orthogonal basis of the Hilbert space $L^{2,\gamma}(\Di)$ as quoted in \cite{Dunkl84,Kanjin2013,Wunsch05}. 

\begin{proposition}[] \label{OrthBasisR}
	The disc polynomials  $ \mathcal{R}_{m,n}^\gamma (z,\bar z)$ form a complete orthogonal system in the Hilbert space $L^{2,\gamma}(\Di)$.
\end{proposition}

Since it is hard to find a rigorous proof of it in the literature, we include a succinct one here.

\begin{proof}[Proof of Proposition \ref{OrthBasisR}]
	By the orthogonality property \eqref{OrthDiscPol}, it is clear that the monomials $\mathcal{R}_{m,n}^\gamma (z,\bar z)$, belong to the Hilbert space $L^{2,\gamma}(\Di)$, under the condition $\gamma >-1$.
	We need to prove completeness of $\mathcal{R}_{m,n}^\gamma (z,\bar z)$ in $L^{2,\gamma}(\Di)$, let $h \in L^{2,\gamma}(\Di)$ such that $S_{m,n}(h):=\scal{h,\mathcal{Z}_{m,n}^\gamma }_{L^{2,\gamma}(\Di)}=0$, for all $m,n$. 
	Then, using polar coordinates as well as the explicit representation of $\mathcal{R}_{m,n}^\gamma$ in terms of the Jacobi polynomials and the change of variable $t=1-2r^{2}$, we can rewrite $S_{m,n}(h)$ as 
	\begin{align*}
	S_{m,n}(h) &= \int^{1}_{-1}\left( \int^{2\pi}_{0}h\left( \sqrt{\frac{1-t}{2}}e^{i\theta}\right) e^{i(m-n)\theta}d\theta\right)  P_{m\wedge n}^{\vert m-n\vert,\gamma}(t)(1-t)^{\vert m-n\vert/2}(1+t)^{\gamma}dt.	 
	\end{align*}
	By discussing the signum of the integer $k=m-n$, we conclude that $S_{m,n}(h) =0$, for all $m,n$, is equivalent to  
	$$\int^{1}_{-1}\left( \int^{2\pi}_{0}h\left( \sqrt{\frac{1-t}{2}}e^{i\theta}\right) e^{ik\theta}d\theta\right)  P_{s}^{\vert k\vert,\gamma}(t)(1-t)^{\vert k\vert/2}(1+t)^{\gamma}dt= 0,$$
	for all nonnegative integer $s$ (=$min(m,n)$).  
	Making use of the Cauchy inequality, we can show that occurring function
	$$t\mapsto (1+t)^{\gamma/2}g_{k}(t); \qquad  g_{k}:t\mapsto \int^{2\pi}_{0}h\left( \sqrt{\frac{1-t}{2}}e^{i\theta}\right) e^{ik\theta}d\theta $$ 
	belongs $L^{2}\left( \left[ -1,1\right],dt\right)$, for which it is known (see e.g. \cite{Koorwin2012}) that the functions $t\mapsto (1-t)^{\vert k\vert/2}(1+t)^{\gamma/2}P_{s}^{(\vert k\vert,\gamma)}(t),$ for varying $s$, form an orthogonal basis. Therefore, 
	$(1+t)^{\gamma/2}g_{k}=0$ and hence $g_{k}=0$ a.e on $[ -1,1]$, for every $k \in \Z$.
	Subsequently, for every $t\in [-1,1]\setminus N$, with $N:= \cup_k \{ t \in \left[-1,1 \right];  g_{k}\neq 0 \}$, we have 	$$\left[  0,2\pi\right] \ni\theta \mapsto h_{t}(\theta)=h\left( \sqrt{\frac{1-t}{2}}e^{i\theta}\right) $$
	belongs to $L^{2}\left( \left[ 0,2\pi\right];d\theta\right)$, by means of Fubini's theorem, and its Fourier transform vanishes at $k$, $\mathcal{F}(h_t) (-k) = g_k(t) = 0$ for $t\in [-1,1]\setminus N$. 
    Thus, the function $h_{t}=0$, a.e. on $\left[ 0,2\pi\right]$, and for almost every $t \in \left[ -1,1\right]$. 
	This completes the proof. 	
\end{proof}

From the explicit expression of the disc polynomials \eqref{expZ}, 
%\textcolor{red}
{one can see that they are a good class of polyanalytic functions} in the unit disc. In fact, 
 $n+1$-analyticity is characterized by those  functions in the form
$f(z,\bz)= \sum\limits_{j=0}^{n} \bz^j \varphi_j(z),$
with $\varphi_j$ being holomorphic functions in $\Di$; and which is equivalent to  be uniquely expressed as \cite{Dolzhenko1994}
$f(z,\bz)= P(z,\bz) + \sum\limits_{j=0}^{n} (1-|z|^2)^j \psi_j(z),$
where $P(z,\bz)$  is a polynomial in $z$ of degree $n$ and $\bz$ of degree at most $n$, and $\psi_j$ are holomorphic in $\Di$.

In the next section, we will explore the crucial role played by these polynomials in describing the so-called weighted true poly-Bergman spaces.

\section{Proofs of main results} 

Throughout this section, $\gamma >-1$ and  $L^{2,\gamma}(\Di)$ is as above. We define the $n$-the weighted poly-Bergman space $\HnWDi_n(\Di) $ to be the space of complex-valued functions $f=\Di \longrightarrow \C$ belonging to $L^{2,\gamma}(\Di)$ and satisfying the generalized Cauchy--Riemann equation $\partial^{n+1}_{\bz} f =0,$ to wit 
$$\HnWDi_n(\Di) :=   ker \partial^{n+1}_{\bz} \cap L^{2,\gamma}(\Di).$$

Theorem \ref{thm1} shows in particular that the introduced  spaces can  equivalently be defined by mean of the disc polynomials $ \mathcal{R}_{m,n}^\gamma $. 

\begin{proof}[Proof of Theorem \ref{thm1}] 
	We begin by proving that the weighted poly-Bergman spaces
	$\HnWDi_n(\Di) $ are closed subspaces of $L^{2,\gamma}(\Di)$
	and spanned by the disc polynomials $\mathcal{R}^{\gamma}_{m,k}$ for varying $m=0,1,2, \cdots$ and $k=0,1,\cdots,n$. To this end, notice first that any $f \in  L^{2,\gamma}(\Di)$ can be expanded as 
	\begin{align}\label{series} f(z) =\sum_{m=0}^\infty \sum_{j=0}^\infty a_{m,j} \mathcal{R}^{\gamma}_{m,j}(z,\bz)
	\end{align} 
	for some complex-valued constants $a_{m,n}$ satisfying the growth condition 
	$$ \norm{f}^2_\gm = 	\sum_{m=0}^\infty\sum_{j=0}^\infty d_{m,j}^\gamma |a_{m,j}|^2 < +\infty .$$
	The series in \eqref{series} is absolutely and uniformly convergent on compact sets of $\Di$.
	Therefore, by straightforward computation, we arrive at  
		\begin{align*} 
		\partial_{\bz}^{k+1}f(z)
& = n! \sum_{m=0}^\infty \sum_{\ell=0}^\infty\frac{ (\alpha+m+1)_{\ell} }{\ell! (\gamma+m+k+2)_{\ell}} a_{m,\ell+k+1}  \mathcal{R}^{\gamma+k+1}_{m,\ell}(z,\bz)  .
	\end{align*}
	This follows since \cite{Koornwinder75,Wunsch05,Aharmim2015} $$\partial_{\bz}^k \mathcal{R}^{\gamma}_{m,j}  = \varepsilon_{j-k}   \frac{n! (\gamma+m+1)_k (\gamma+1)_{m+j-k} }{(j-k)! (\gamma+1)_{m+j}}\mathcal{R}^{\gamma+k}_{m,j-k},$$ where $\varepsilon_{j-k}=1$ when $j\geq k$ and $\varepsilon_{j-k}=0$  otherwise. 
Subsequently, the coefficients $a_{m,\ell+k+1}$ are closely connected to the Fourier coefficients of $\partial_{\bz}^{k+1}f \in L^{2,\gamma+k+1}(\Di)$ (see e.g. \cite{Kanjin2013}). More precisely, we have 
$$ a_{m,\ell+k+1} 
= \frac{\ell! (\gamma+m+k+2)_{\ell}}{ n!(\alpha+m+1)_{\ell}  d_{m,l}^{\gamma+k+1}} \scal{ \partial_{\bz}^{k+1}f, \mathcal{R}^{\gamma+k+1}_{m,\ell}}_{\gamma+k+1} 
.$$
Thus,  $f\in\HnWDi_n(\Di) $ is equivalent to $a_{m,\ell+k+1} =0$ for any nonnegative integers $m,\ell$. This infers 
	\begin{align}\label{series2} f(z) = \sum_{m=0}^\infty\sum_{j=0}^{n} a_{m,j} \mathcal{R}^{\gamma}_{m,j}(z,\bz)
\end{align}
in $L^{2,\gamma}(\Di)$ and hence
$$\HnWDi_n(\Di) = \overline{Span\{\mathcal{R}^{\gamma}_{m,j}; m=0,1, \cdots, \, j=0,1, \cdots,n \}}^{L^{2,\gamma}(\Di)}.$$
Thus, $\HnWDi_n(\Di)$ is clearly closed subspace of $L^{2,\gamma}(\Di)$.
This completes the proof of Theorem \ref{thm1}.
\end{proof}

\begin{proof}[Proof of Theorem \ref{thm2}] 
	An immediate consequence of the above discussion in the proof of Theorem \ref{thm1},  we claim that  
	the orthogonal Hilbertian decompositions 
	$$\displaystyle L^{2,\gamma}(\Di) =  \bigoplus_{n=0}^\infty \AnTWDi_{n}(\Di) \quad  \mbox{and} \quad \HnWDi_n(\Di) = \bigoplus_{k=0}^n \mathcal{A}_{k}^{2,\gamma}(\Di)$$
	hold trues. We need only to prove that the family of functions $\mathcal{R}^{\gamma}_{m,n}$; $m=0,1,2, \cdots,$ form an orthogonal basis of $\AnTWDi_{n}(\Di)= \HnWDi_n(\Di) \ominus \HnWDi_{n-1}(\Di)$ 
	Indeed, by considering 
	$$ \mathcal{B}^{2,\gamma}_n(\Di):= 
	\overline{Span  \left\{ \mathcal{R}^{\gamma}_{m,n} ; \, m=0,1,2, \cdots \right\}}^{L^{2,\gamma}(\Di)},$$
	it is clear that they form an orthogonal sequence of Hilbert spaces in $L^{2,\gamma}(\Di)$. The orthogonality follows from the orthogonality property of disc polynomials. Moreover, we claim that 
	$\HnWDi_n(\Di) = \bigoplus_{k=0}^n \mathcal{B}^{2,\gamma}_k(\Di)$ thanks to Theorem \ref{thm1} and $\mathcal{B}^{2,\gamma}_n(\Di)\subset  \AnTWDi_{n}(\Di) $ which holds by induction. Hence $\mathcal{B}^{2,\gamma}_n(\Di)=  \AnTWDi_{n}(\Di)$. The fact that 
	$\mathcal{A}_{n}^{2,\gamma}(\Di)$ is a reproducing kernel Hilbert space follows mainly using Riesz theorem on the representation for the linear mapping $f \longmapsto f(z_0)$, for fixed $z$, being  continuous. Subsequently, there exists some element $ K^\gamma_{n,z_0} = K^\gamma_{n}(\cdot,z_0)\in \AnTWDi_{n}$ such that $ f(z) = \scal{f , K^\gamma_{n,z_0}}_{\gamma}$ for every $f \in \AnTWDi_{n}$.
	In fact,  any $f\in \AnTWDi_{n}(\Di)$, can be expand as $f(z) = \sum_{m=0}^\infty  a_{m} \mathcal{R}^{\gamma}_{m,n}(z,\bz)$. Then, by Cauchy Schwartz inequality and that fact $\mathcal{R}^{\gamma}_{m,n}(z,\bz)\leq 1$, we get  
	\begin{align*} 
	|f(z)| \leq  \sum_{m=0}^\infty  |a_{m}| 
	\leq  c \left(  \sum_{m=0}^\infty  d_{m,n}^\gamma |a_{m}|^2\right)^{1/2}  \leq c \norm{f}_{\gm}
	\end{align*}
	with 
	\begin{align*} 
	c :=  \left(  \sum_{m=0}^\infty \frac{1}{d_{m,n}^\gamma} \right)^{1/2}  =   \left( \frac{(\gamma+n+1) (\gamma+1)_n}{\pi n! }	 {_2F_1}\left( \begin{array}{c} \gamma+n+2, \gamma+1 \\ \gamma+n+1 \end{array}\bigg | 1\right)\right)^{1/2}. 
	\end{align*}
	This completes  our check of $(i)$ and $(ii)$. 
	
	To prove $(iii)$, we know from general theory of reproducing kernel separable Hilbert spaces, the reproducing kernel function $ K^\gamma_{n}$ is uniquely determined by 
	$$ K^{\gm}_n(z,w)=	\sum_{m=0}^{+\infty}  
  \psi_{m}(z)  \overline{\psi_{m}(w)}  $$
for any complete orthonormal $\psi_{m}$ of $\AnTWDi_{n}$.  
Thus, the closed formula of $K^{\gm}_n(z,w)$ in terms of successive derivatives of special function follows by direct computation using the disc polynomials $\mathcal{R}^{\gamma}_{m,n}$. In fact, by means of $\overline{\mathcal{R}^{\gamma}_{m,n}(w,\bw)}  = \mathcal{R}^{\gamma}_{m,n}(\bw,w)$ as well as the observation that the disc polynomials can be realized as 
	\begin{align} \label{ROpF1} 
	\mathcal{R}_{m,n}^{\gamma}(z,\bz) = \frac{(-1)^m}{(\gamma+1)_m} (1-|z|^2)^{-\gamma}
	\partial_{z}^{n} \left( z^m (1-|z|^2)^{\gamma+n} \right) 
	\end{align}
	and 
		\begin{align} \label{ROpF12} 
	\overline{\mathcal{R}_{m,n}^{\gamma}(w,\bw)} = \frac{(-1)^m}{(\gamma+1)_m} (1-|w|^2)^{-\gamma}
	\partial_{\bw}^{n} \left( \bw^m (1-|w|^2)^{\gamma+n} \right) ,
	\end{align}
 we get
	\begin{align*} 	
K^{\gm}_n(z,w)&=	\sum_{m=0}^{+\infty}  
\frac{\mathcal{R}^{\gamma}_{m,n}(z,\bz)  \overline{\mathcal{R}^{\gamma}_{m,n}(w,\bw)} }{ d_{m,n}^\gamma }      
 \\ &= 
\frac{(1-|z|^2)^{-\gamma}(1-|w|^2)^{-\gamma}}{[(\gamma+1)_n]^2}	   \partial_{z}^{n} \partial_{\bw}^{n}\left\{(1-|z|^2)^{\gamma+n}  (1-|w|^2)^{\gamma+n} \Xi_n^\gamma(z,w) \right\} 
 	\end{align*}
 with
 	\begin{align*}
\Xi_n^\gamma(z,w) &:=  \sum_{m=0}^{+\infty}\frac{z^m \bw^m }{ d_{m,n}^\gamma } = \frac{(\gamma+1)_n}{\pi n! } 	
  \sum_{m=0}^{+\infty}  
  (\gamma+m+n+1) (\gamma+1)_m  \frac{ z^m \bw^m }{  m!  }  
 	 \\&=	\frac{(\gamma+n+1) (\gamma+1)_n}{\pi n! }	 {_2F_1}\left( \begin{array}{c} \gamma+n+2, \gamma+1 \\ \gamma+n+1 \end{array}\bigg | z\bw \right) .  
 		\end{align*} 
 The last equality follows making use of the fact that 
 $$(\gamma+m+n+1)=\frac{(\gamma+n+1)_m(\gamma+1)_n}{(\gamma+1)_m}, $$
 keeping in mind the definition of the hypergeometric function
 $ {_2F_1}$. This completes the check of  the  closed expression in \eqref{RK-wBsm}.
\end{proof}

\begin{remark}
The introduced spaces are closely connected to the spectral theory of the magnetic Laplacian 
$ \mathfrak{L}_\nu  = -\Delta  
- \nu N + \nu^2 |z|^2$, seen as
self-adjoint and elliptic second order differential operator \cite{Zhang92,GhIn2005}.
It is  essentially the Laplace--Beltrami operator
%$\Delta:=(1 - |z|^2)^2 \partial_z\partial_{\bz}$ perturbed by the rotation operator $ N =  (1 - |z|^2)\left(z \partial_z- \bz \partial_{\bz}\right).$
$$\Delta:=(1 - |z|^2)^2\frac {\partial^2} {\partial z\partial \bar z}$$
perturbed by the rotation operator
$$ N =  (1 - |z|^2)\left(z\frac {\partial}{\partial z}- {\bar z}\frac{\partial}{\partial{\bar z}}\right).$$
Geometrically, this is the magnetic Schr\"odinger operator \cite{GhIn2005}
$ %\begin{align*}% \label{LH}
\mathfrak{L}_\nu = (d+i\nu ext(\theta))^{*}(d+i\nu ext(\theta)),
$ %\end{align*}
associated to the hyperbolic geometry on the unit disc endowed with its Poincaré metric.
%$$ds^2 := \frac{1}{(1 - |z|^2)^2}dz\otimes d\bar z.$$
Here $d$ and $ext(\theta)$ are respectively the differential operator and the exterior multiplication by the differential $1$-form ( potential vector) $\theta(z) = (\partial -\overline{\partial}) \log (1-|z|^2)$. 
The adjoint operation is taken with respect to the Hermitian scalar product on compactly supported differential forms 
%\begin{align*} %\label{sp}
%(\alpha,\beta)&:=\int_{D}\alpha\wedge\star \beta,
%\end{align*}
involving the Hodge star operator canonically associated to the hyperbolic metric. 
\end{remark}

%\section{Asymptotic behavior}

%\newpage

%\section{Proofs of main theorems}

%\newpage

\end{document}